\documentclass[12pt]{article}
\newcommand{\comment}[1]{}

\newcommand{\reals}{{\bf R}}

\newcommand{\supp}{\mathop{\rm supp}\nolimits }

\newcommand{\BibTeX}{{\rm B\kern-.05em{\sc i\kern-.025em b}\kern-.08em     
    T\kern-.1667em\lower.7ex\hbox{E}\kern-.125emX}}

\usepackage{geometry}
\usepackage{hyperref}

\providecommand\note{}
\renewcommand\note[1]{}
\renewcommand\marginpar[1]{}

\usepackage{amsmath,amsthm,amscd,amssymb}

\usepackage{chngcntr}
\usepackage{apptools}

\AtAppendix{\counterwithin{equation}{section}}

\renewcommand{\reals}{\mathbb{R}}

\newcommand{\intr}{\mathop{\rm int}\nolimits}

\numberwithin{equation}{section} 
\newtheorem{theorem}[equation]{Theorem}
\newtheorem{proposition}[equation]{Proposition}

\newtheorem{lemma}[equation]{Lemma}

\begin{document}
\title{Estimates for Brascamp-Lieb forms in $L^p$-spaces with
  power weights\footnotetext{2010 MSC: 26B15, secondary: 52B99}}


\author{
 R.M.~Brown\footnote{
Russell Brown   is  partially supported by  grants from the Simons
Foundation (\#195075,\#422756).
} \\ Department of Mathematics\\ University of Kentucky \\
Lexington, KY 40506-0027, USA
\and
C.W.~Lee 
\\ Department of Mathematics\\ University of Kentucky \\
Lexington, KY 40506-0027, USA
\and
K.A.~Ott\footnote{Katharine Ott is partially supported by a grant from the Simons
Foundation (\#526904).} 
\\Department of Mathematics\\Bates College\\Lewiston, ME  04240-6048,
USA
}

\date{}

\maketitle

\begin{abstract}
  We study a family of Brascamp-Lieb forms acting on families of  weighted $L^p$-spaces and Lorentz spaces  where the weight is a power of the distance to the origin. We establish a set of necessary conditions and a set of sufficient conditions for the finiteness of these forms 
in Lorentz spaces.  The conditions are close to optimal. 
\end{abstract}

\section{Introduction}\label{Introduction}

A family of multi-linear forms was studied  by Brascamp and Lieb \cite{MR0412366} and these forms have found many applications since then. 
In previous work by the authors and collaborators \cite{MR1871279,MR3494243,MR2754810,MR3551174}, estimates for
Brascamp-Lieb forms are used  to study a non-linear
scattering map in two dimensions. 
The most recent paper in this series
\cite{MR3494243} proves estimates for Brascamp-Lieb forms in weighted
$L^2$-spaces where the weight is a power of the distance to the
origin. While the estimates established in section 3 of
\cite{MR3494243} were sufficient for the problem at hand, this work
leaves open the question of finding the largest set of indices for
which we have weighted estimates for a Brascamp-Lieb form.  Our goal
in this paper is to show that the techniques developed in
\cite{MR3494243} can be used to give close to optimal conditions for a
special class of forms. 

To illustrate the utility of the current results, we provide new proofs of
known results on certain linear and multi-linear fractional integrals
and establish  a large family of new results.  The earlier works that we
generalize include a result of Stein and Weiss \cite[Theorem
  B${}^*$]{MR2754810}, multi-linear generalizations studied by
Grafakos \cite[Theorem 1]{MR1164632}, and a result of Komori-Furuya
\cite[Theorem 3]{MR4047697}. There are still many interesting, open
problems related to general Brascamp-Lieb forms in weighted spaces.

The forms we will study are defined using 
a collection of non-zero vectors $\{ v_1, \dots, v_N\}$   in $ \reals ^2$.
We assume that this collection satisfies the following property: 
\begin{equation} \label{Assumption}
  \mbox{Each pair of distinct vectors in $\{v_1, \dots, v_N\}$ is a basis for $\reals^2$.} 
\end{equation}
For $v = (v^1, v^2) \in \reals ^2$ and $x=(x^1,x^2) \in \reals ^ { 2k}$, we define $ v\cdot x= v^1 x^1 +  v^2 x^2$. 
We define a multi-linear form by 
\begin{equation}\label{form}
\Lambda(f_1,\dots, f_N)=\int_{\mathbb{R}^{2k}}f_1(v_1\cdot x)\dots f_N(v_N\cdot x)\, dx,
\end{equation}
where $f_1,\dots,f_N$ are  functions on $\mathbb{R}^k$.
We want to bound the form \eqref{form} using the $L^{p_j}_{\lambda_j}$-norms of the functions
$f_j$. The space $L^p_{\lambda}$ is the collection
$
L^p_{\lambda}=\{f: |\cdot|^\lambda f\in L^p\}
$
where we allow $1/p\in [0,1]$ and $\lambda\in\mathbb{R}$.
The norm in $L^p_{\lambda}$ is given by 
$$
\|f\|_{L^p_{\lambda}}=\| |\cdot|^{\lambda}f\|_{L^p}.
$$
All norms will be over $\reals^k$.

Our aim is to determine the collection  of indices
$(1/p_1, \lambda _1, \dots, 1/p_N, \lambda _N) $  for  which the estimate 
\begin{align}\label{goal}
|\Lambda(f_1,\dots,f_N)| \leq C\prod_{j=1}^N \|f_j\|_{L^{p_j}_{\lambda_j}},
\end{align}
holds for a finite constant $C$.  Throughout this paper we will assume the functions $f_j$ are non-negative. Once we establish \eqref{goal} for non-negative functions, the estimate will follow for real or complex valued functions. 

The techniques that we use can be traced back at least to O'Neil \cite{MR0146673}, who studied fractional
integration on $L^p$-spaces by observing that the kernel of the Riesz potential, $|\cdot|^{-\lambda}$, 
lies in a weak $L^p$-space, or Lorentz space. We begin with a characterization of the families of (unweighted) $L^p$-spaces where the  form \eqref{form} is bounded. The first characterization of families of $L^p$-spaces where the form is bounded is  due to Barthe 
 \cite[Proposition 3]{MR1650312}.  Additional progress has been made by  Carlen, Lieb and Loss \cite{MR2077162} and
Bennett, Carbery, 
Christ and Tao \cite{MR2661170}.   Estimates for Brascamp-Lieb forms in Lorentz spaces were given  by Nie and Brown \cite{MR2754810} and especially Christ in an appendix to \cite{MR3551174}.   We use that  $ |\cdot |^ { -\lambda}$ lies in a Lorentz space for $ \lambda >0$  and a version of H\"older's inequality  to  give  estimates for the form \eqref{form} when $f_j$ belong to weighted spaces $L^p_\lambda$  with $\lambda >0$. The novel ingredient in this paper and the step that restricts our work to forms satisfying \eqref{Assumption} is an algebraic argument that allows us to obtain estimates for the form \eqref{form} with $f_j \in L^p_\lambda$ with $\lambda <0$ in terms of estimates where $ \lambda \geq 0$. 

To begin our development, we 
let $\mathcal{P}$ denote the set of indices where \eqref{goal} holds, 
\begin{align}\label{setP}
\mathcal{P}=\{(\frac{1}{p_1},\lambda_1, \dots , \frac 1{p_N}, \lambda _N ) :\eqref{goal}
\mbox{ holds }\}\subset([0,1]\times \mathbb{R})^N,
\end{align}
and we have the following observation.

\begin{proposition}\label{convex}
The set $\mathcal{P}$ is  convex.
\end{proposition}

\begin{proof}
  According to a result in Bergh and L\"ofstr\"om \cite[Theorem
    5.5.3]{MR0482275}, the spaces $ L^p_\lambda$ form a complex
  interpolation scale. Thus, 
the proposition follows from a multi-linear version of the
Riesz-Thorin interpolation theorem in Bergh and L\"ofstr\"om
\cite[Theorem 4.4.1]{MR0482275}.
\end{proof}
We will see in Theorem \ref{nec} below that $\mathcal{P}$ lies in the hyperplane
\begin{align}
\left\{(\frac{1}{p_1},\lambda_1 , \dots \frac  1 {p_N}, \lambda _N ) :
\sum_{j=1}^N(\frac{1}{p_j}+
\frac {\lambda_j} k ) =2\right\}.
\end{align}
The goal of this note is to characterize the 
closure of $\mathcal{P}$ 
in this hyperplane.

We now give the main results of this paper. 
Theorem \ref{suff} gives conditions 
that imply the estimate \eqref{goal} and  Theorem \ref{nec} gives necessary conditions
for \eqref{goal}. Except for allowing equality in the condition
(\ref{HyperStrict}) and the inclusion of the
endpoints $1/p_j \in [0,1]$, the conditions of
Theorem \ref{nec} are identical to those of
Theorem \ref{suff}.

\note{included condition that $\infty >  p _j > 1$ on 5/12.}
\begin{theorem}\label{suff}
Suppose that $({1}/{p_1},\lambda _1, \dots, 1/p_N, \lambda _N) 
\in ((0,1)\times\mathbb{R})^N$ and that the following list of conditions are true:
\begin{align}
& \sum_{j=1}^N(\frac{1}{p_j}+\frac {\lambda_j} k) =2,\label{Scaling}\\
& \sum_{j\neq\ell}(\frac{1}{p_j}+\frac {\lambda_j} k)>1,\quad\ell=1,\dots,N,\label{HyperStrict}\\
& \sum_{j\neq\ell}\lambda_j\geq 0,\quad\ell=1,\dots,N,\label{Index}\\
& \sum_{j=1}^N \frac{1}{p_j}\geq 1. \label{Interpolation}
\end{align}
Then the estimate \eqref{goal} holds.
\end{theorem}


\begin{theorem}\label{nec}
Suppose that \eqref{goal} holds. Then the set of indices $(1/p_1, \lambda _1, \dots , 1/ p_N , \lambda _N ) 
\in ([0,1]\times\mathbb{R})^N$ satisfy \eqref{Scaling}, \eqref{Index},
\eqref{Interpolation}, and the inequalities
\begin{align}
& \sum_{j\neq\ell}(\frac{1}{p_j}+\frac {\lambda_j} k)\geq 1,\quad\ell=1,\dots,N.\label{Hyper}
\end{align}
\end{theorem}

The proofs that the conditions \eqref{Scaling} and \eqref{Hyper} are necessary   are small extensions of arguments in Bennett, Carbery, Christ, and Tao \cite{MR2661170}. Condition \eqref{Interpolation} or equivalent statements for multi-linear operators are known to arise when studying the real method of interpolation for multi-linear operators. Our argument to establish the necessity of \eqref{Interpolation}  is similar to the proof of a  recent result of Bez, Lee, Nakamura, and Sawano \cite[Theorem 2]{MR3665018}  on Brascamp-Lieb forms in Lorentz spaces. 
The condition \eqref{Index} is closely tied to the weighted spaces we are using. The necessity of this condition appears to be new. In section \ref{Apps} we note that similar conditions have arisen in the study of multi-linear fractional integrals in the spaces $L^p_\lambda$.

The proofs of Theorems \ref{suff} and
\ref{nec} are presented in the next  two sections of the paper. 
In the last section, we show how to obtain the 
result of Stein and Weiss on fractional integration
in weighted $L^p$-spaces, discuss generalizations of this result for multi-linear fractional integrals, and present a few examples illustrating the limitations
of our results.

{\em Acknowledgement: }We thank the referee for their careful reading of the manuscript and many useful suggestions. 

\section{Proof of Theorem \ref{suff}} \label{suffproof}

Our first step  is to study the estimate  \eqref{goal} in unweighted $L^p$-spaces. We will show that for unweighted spaces, the set of indices $(1/p_1, \dots,1/p_N)$ for which the estimate \eqref{goal} holds is independent of the dimension  $k$ appearing in the definition of the form. 
To give the proof, we temporarily will use the notation $ \Lambda _k$ for the form \eqref{form} when the integration is on $ \reals ^ { 2k}$. 

\begin{proposition} \label{KEquals1} Fix a vector of indices $ (1/p_1, \dots , 1/ p_N)$ and assume that $ \lambda _1 = \dots = \lambda _N=0$. Let $C_k = C_k ( 1/ p _1, \dots, 1 /p_N)$ be the constant in \eqref{goal} for the form $ \Lambda _k$. We have $ C_k= (C_1)^k$. 
\end{proposition}

\begin{proof}   We first show $ (C_1)^k \leq C_k$.   If $ f _ j$ is a non-negative measurable  function on $ \reals$,  we define $ F_j: \reals ^k \rightarrow [0, \infty)$ by $ F_j(y_1, \dots, y _k) = \prod _ {\ell =1 } ^ k f_j ( y_ \ell)$ and observe that $ \|F_j\| _ { L^ p ( \reals ^ k)} = \|f_j\|^k_ { L^ p( \reals ) }$. Thus, if we have \eqref{goal} for $ \Lambda _k$, we conclude that 
        $$
        \Lambda_1 ( f_1, \dots, f_N) ^ k = \Lambda _k (F_1, \dots , F_N)\leq C_k \prod _{ j =1} ^ N \|F_j \| _ { L^ { p_j }( \reals ^k )} = C_k  \prod _{ j =1 } ^ N \| f _j \| ^k_ { L^ { p_j } ( \reals ) } .
        $$
        This implies $ (C_1) ^ k \leq C_k$.

To show the reverse inequality,  it suffices to show $ C_{ k +1} \leq  C_1 C_k$ for $k=1,2,\dots$. Towards this end, let $ F_j$ be a non-negative, measurable function on  $ \reals ^ { k +1}$ and set $f_j(y_1, \dots, y_k) = \| F_j (y_1, \dots, y_k , \cdot ) \| _ { L^ p ( \reals) }$ where the norm is taken in the last variable,
$y _ {  k +1 }$. An application of Tonelli's theorem and the estimate \eqref{goal} for $k=1$ gives that
  $$ \Lambda _ { k+1} (F_1, \dots , F_N)  \leq  C_1 \Lambda _k (f_1, \dots , f_N). $$
  It follows that  $ C_{ k+1} \leq C_1 C_k$ and then an induction argument gives $ C_k \leq ( C_1)^k$.
\end{proof}

Our next step requires us to consider Lorentz spaces
$L^{p,r}$, for $ 1\leq p,r\leq \infty$, as defined in
\cite{MR0482275}. We will also utilize weighted Lorentz spaces, $L^
     {p,r}_ \lambda$, equipped with the quasi-norm $ \| f \|_ { L^ { p,r
       } _\lambda} = \| |\cdot |^ \lambda f\|_ { L^ { p,r}}$.
As is well-known, Lorentz spaces arise naturally in real
interpolation and a  multi-linear interpolation theorem is 
 important for our argument.      


\begin{theorem}
  \label{BCCTCor}
  Suppose that $ (1/p_1, \dots, 1/p_N) \in (0,1)^N$, satisfies
  \eqref{Scaling} (with all $ \lambda _j =0$) and $ (1/r_1, \dots, 1/r_N)$ satisfies
  $ \sum _ { j=1} ^N \frac 1 { r _ j } \geq 1$. Then there exists a finite constant $C$ so that
\begin{equation}  \label{GoalLorentz}
  \Lambda ( f_1, \dots, f_N) \leq C \prod _{j=1} ^N \| f_j \| _ { L^ { p_j, r _j }}.
\end{equation}
\end{theorem}
Our proof follows the argument in 
Christ's appendix to the work of Perry  \cite{MR3551174}. 
 We repeat   Christ's argument in order to show that the conditions are independent of the exponent $k$ appearing in the definition of the form.
Additional information on the multi-linear interpolation result used below  can be found in 
Christ \cite[pp.~227--228]{MR766216} or 
 Janson \cite{SJ:1986}. A more recent discussion may be found in Grafakos and Kalton \cite{MR1812822}.

\begin{proof} 
 Throughout this proof, we assume all  $\lambda _j =0$. Under
 this assumption, let $ \mathcal{P}_0$ represent the collection of exponents $(1/p_1, \dots, 1/p_N)$ for which \eqref{goal} holds. We note that Proposition \ref{KEquals1} implies that the set $\mathcal{P}_0$ is independent of $k$ and thus it suffices to consider the case $k=1$.   
  We claim that $ \mathcal{P}_0$ is the collection of  $(1/p_1, \dots, 1/p_N)\in [0,1]^N$ which satisfy \eqref{Scaling} and \eqref{Hyper}. 
  From Theorem 2 in the work of Bennett {\em et. al.\ } \cite{MR2661170} we have that $\mathcal{P}_0$ is
  characterized by the equality \eqref{Scaling} and the family of inequalities 
  \begin{equation} \sum _ { j=1} ^N \frac { \dim (v_j \cdot V ) } { p _j } \geq \dim(V)\label{SubSpace}
  \end{equation}
  for each subspace $V \subset  \reals ^2$. Here, we are using $ v_j \cdot V = \{ v_j \cdot w : w \in V\}$.  If $ V = \reals ^2$, then \eqref{SubSpace} follows from \eqref{Scaling}. If the dimension of $V$ is one, then thanks to the assumption \eqref{Assumption}, we have $ \dim(v_\ell \cdot V ) = 0$ for at most one $\ell$. If $\dim(v_j \cdot V) =1$ for all $j$, then \eqref{SubSpace} follows from scaling again. If $ \dim (v_\ell \cdot V) =0$ for one $\ell$, then we have
  $$ \sum  _ { j \neq \ell }  \frac 1 { p _j } = 2 - \frac 1 { p _ \ell } .$$
  Thus \eqref{SubSpace} follows from our assumption that $ 1/p_\ell \in [0,1]$.

At this point, we appeal to the real interpolation  argument in Christ's Appendix 1 of \cite{MR3551174}, which gives that the estimate \eqref{GoalLorentz} holds  in the interior of $\mathcal{P}_0$ provided the indices $r_j$ satisfy $ \sum _ { j =1 } ^ N 1/r _ j \geq 1$.  The  theorem follows. 
\end{proof}  


Now we turn to the study of the form \eqref{form} in 
weighted $L^p$-spaces. To this end we will need a version
of H\"older's inequality in Lorentz spaces. One proof of
the following proposition may be found in O'Neil \cite[Theorem 3.4]{MR0146673}.

\begin{proposition} \label{LorentzHolder}
Let $f_1$ and $f_2$ be measurable functions on $\mathbb{R}^k$. There exists a finite constant $C$
such that 
\begin{align}\label{HIB}
\|f_1 f_2\|_{L^{p,r}}\leq C\|f_1\|_{L^{p_1,r_1}}\|f_2\|_{L^{p_2,r_2}},
\end{align}
provided $\frac{1}{p} = \frac{1}{p_1}+\frac{1}{p_2}<1$, $\frac{1}{r} \leq 
\frac{1}{r_1}+\frac{1}{r_2}$, and $\frac{1}{r_1}, \frac{1}{r_2}$, and 
$\frac{1}{r}$ lie in $[0,1]$.
\end{proposition}

\note{ There is an endpoint result, $\|fg\|_1 \leq C 
  \|f\|_{p,r}\|g\|_{p',r'}$, but this is not helpful for us. Should we
include it? }
We also need the easily verified  fact that if  $\lambda/k = 1/r$, then for 
$\frac{1}{r}\in [0,1]$, 
\begin{equation}\label{LPow}
\| |\cdot|^{-\lambda}\|_{L^{r,\infty}}\leq C.
\end{equation}
%
Combining (\ref{LPow}) with the H\"older inequality \eqref{HIB} gives
\begin{equation} \label{LHCons}
\|f\|_{L^{r,q}} \leq C
  \|f\|_{L^{p,q}_ \lambda }, \qquad \frac 1 p + \frac \lambda  k  \in (0,1), \ \frac 1 p , \frac \lambda k \in (0,1), \ \frac 1 q \in [0,1]. 
\end{equation}
\note{ I am not sure what we finally decided by including $ \reals
  ^k$.   We could include that the estimate is for $\reals ^k$ in the
  text and be consistent. Or decide that ``A foolish consistency is
  the hobgoblin\dots''
  
  Katy added a sentence in the intro about norms in $\reals^k$ following
  \eqref{normwlp}.}

The next Lemma is rather technical, so we will try to explain its role
in the proof of Theorem \ref{suff}. Using Theorem \ref{BCCTCor}  
and \eqref{LHCons}, we will be
able to establish  \eqref{goal} in the case where the weights satisfy 
$\lambda _j \geq 0$ for $j = 1,\dots,N$.  Lemma \ref{Step} 
allows us to make a reduction to non-negative exponents. In the case 
where one or more of the exponents $
\lambda _j$ is negative, we  use linear relations among the
vectors $ v_j$ to relate an instance of the estimate \eqref{goal} in
spaces where the weight has a negative exponent to a family of
estimates in spaces where the weights all have  non-negative  exponents. 

%
%
In the  next Lemma and throughout the remainder of this paper, we will denote the support of a vector by $ \supp
( \lambda _1, \dots, \lambda _N )  = \{ j : \lambda _j \neq 0\}$. We
also use $ \lambda ^ +$ to denote the positive part of a real number $
\lambda$. 
\begin{lemma} \label{Step} Suppose that the indices $(\lambda _1, \dots, \lambda _N) $  satisfy \eqref{Index}. Then there exists a family
  of indices $\{ \alpha ^ \ell\} = \{ (\alpha _1 ^ \ell, \dots, \alpha_N^ \ell) : \ell =1 ,\dots, L\}$ so that
  \begin{gather}\label{E}
    \Lambda (f_1, \dots, f_N) \leq C \sum _ { \ell =1 } ^L \Lambda (
    |\cdot |^ { \alpha _1 ^ \ell } f_1, \dots, |\cdot  | ^ { \alpha _N^
      \ell} f_N), \\
    \label{P}
    0\leq \lambda _j -\alpha _j ^ \ell,  \qquad
    j=1, \dots, N,\ \ell = 1, \dots, L, \\
        \label{P1}
     \lambda _j -\alpha _j ^ \ell \leq \lambda _j ^ +, \qquad
    j=1, \dots, N,\ \ell = 1, \dots, L, \\
    \label{Z} \sum _ { j=1 } ^ N \alpha _j ^ \ell = 0, \qquad \ell =
    1, \dots, L. 
  \end{gather}
\end{lemma}
\begin{proof} We begin with the collection consisting of one element,
  $ \alpha ^ 1 = (0,\dots, 0)$ and observe that the conditions
  \eqref{E}, \eqref{P1}, and \eqref{Z} are clear. We also note that since $(
  \lambda_1, \dots,\lambda _ N)$ satisfies \eqref{Index}, then
\begin{equation}
  \label{Index2} \sum _ { j \neq m }  ( \lambda _j - \alpha _j ^ \ell)
  \geq 0, \qquad m=1,\dots, N, \ \ell = 1, \dots,L
\end{equation}
    holds for  this collection.  However, the condition \eqref{P} will
  fail if some $ \lambda _j <0$. If the condition \eqref{P} fails and
  we have $\lambda _{j_0}- \alpha _ { j_0}^ \ell  < 0$ for some $ j_0$
  and $\ell$, we will replace $ \alpha ^ \ell$ with two vectors $
  \beta$ and $ \gamma$ so that \eqref{E}, \eqref{P1}, and \eqref{Z} continue to
  hold for the new family. We will also have that the sets $ \supp (
  \lambda  - \beta ) $  and $\supp ( \lambda  -\gamma) $ are proper 
  subsets of $ \supp ( \lambda  -\alpha^\ell )$. This guarantees that the substitution
  procedure will eventually terminate and we will obtain a collection
  that satisfies \eqref{P}.

  To describe the substitution step, let $ \alpha = (\alpha_1, \dots, \alpha _N ) $ satisfy \eqref{P1}, \eqref{Z}, and \eqref{Index2},  but suppose
  that there is an index $j_0$ so that $ \alpha _{ j _0 } >
  \lambda _ { j_0}$. We claim that: (1) We may find two indices $j_1$
  and $j_2$ so that $ \alpha _{ j _m } < \lambda _{ j _m } $,
  $m=1,2$; (2) We may replace $ \alpha$ by two vectors $ \beta$ and $
  \gamma$ so that
  \begin{equation}\label{Progress}
  \Lambda (|\cdot |^ { \alpha_1 } f_1, \dots, |\cdot |^ { \alpha _N }
  f_N)
  \leq C(
  \Lambda (|\cdot |^ { \beta_1 } f_1, \dots, |\cdot |^ { \beta _N }
  f_N) +
  \Lambda (|\cdot |^ { \gamma_1 } f_1, \dots, |\cdot |^ { \gamma _N }
  f_N)) 
  \end{equation}
  with $\supp( \lambda _1 -\beta _ 1, \dots, \lambda _N - \beta_N), \supp ( \lambda _1 -\gamma  _1, \dots, \lambda _N-\gamma _N) \subset \supp ( \lambda _1 - \alpha _1, \dots, \lambda _N- \alpha _N) \setminus \{ j _ 0 \} $.
In addition, the vectors $
  \beta $ and $\gamma $ satisfy \eqref{P1}, \eqref{Z}, and \eqref{Index2}.

  To establish (1), we take the average of the $N$ inequalities in
  \eqref{Index2} and obtain that $ \sum _ { j=1} ^N ( \lambda _j
  -\alpha _j) \geq 0$. Thus, if there is one index $j_0$ with $
  \lambda _{ j _0 }-\alpha _ { j _0} < 0$, then there must be an index
  $j_1$ with $ \lambda _{ j _1} -\alpha _ { j _1}>0$. Now, we may use
  the condition \eqref{Index2} with $ m = j _1$ to find $j_2$ with 
$ j_2 \neq j_1$ and $ \lambda _{ j_2} - \alpha _ { j
    _2} >0$.  Recalling assumption \eqref{Assumption},  the set $ \{ v_{j_1} , v_{ j _2} \}$ forms a 
  basis for $ \reals ^2$ and  we may write $ v_{j_0} = c _1 v_ { j_1} +
  c_2 v_{ j_2}$. Since the index $ \alpha _{ j _0 } - \lambda _{ j _0
  } >0$, we have
  $$
  |v_{ j_0 } \cdot x | ^ { \alpha _ { j _0 } - \lambda_ { j _ 0 }}
  \leq C(
  |v_{ j_1 } \cdot x | ^ { \alpha _ { j _0 } - \lambda_ { j _ 0 }}
  +
  |v_{ j_2 } \cdot x | ^ { \alpha _ { j _0 } - \lambda_ { j _ 0 }}).
  $$
  Substituting this inequality into the form $ \Lambda$  gives
  \eqref{Progress} with $\beta$ and $ \gamma$ defined by
    \begin{align*}
    \beta _ {j _ {\hphantom{0}}}& = \alpha _j ,  \qquad \mbox{if } j \neq j _0, j_1\\
    \beta _ { j _ 0 } & = \lambda _ { j _ 0 } = \alpha _ { j _ 0 }
    +\lambda _ { j _0} - \alpha _ { j _0 } \\
        \beta _ { j _ 1 } &  = \alpha _ { j _ 1 }
        - (\lambda _ { j _0 } - \alpha _{ j_0}), \\
  \end{align*}
  and
    \begin{align*}
    \gamma _ {j _ {\hphantom{0}}}& = \alpha _j ,  \qquad \mbox{if } j \neq j _0, j_2\\
    \gamma _ { j _ 0 } & = \lambda _ { j _ 0 } = \alpha _ { j _ 0 }
    +\lambda _ { j _0} - \alpha _ { j _0 } \\
        \gamma _ { j _ 2 } &  = \alpha _ { j _ 2 }
         - (\lambda  _ { j _0 }- \alpha _ { j_0}). \\
    \end{align*}
As $\lambda _j- \beta_j= \lambda _j - \alpha_j$ for $j\neq j_0, j_1$,  $ \lambda _{ j _0} - \beta _{j _0} =0$ and $\lambda _ { j _1 } - \beta _ { j_1} < \lambda _ { j_1} - \alpha_ { j _1}$, it follows that $ \beta$ satisfies \eqref{P1} and similarly $ \gamma $ satisfies \eqref{P1}. 
    As $ \lambda _{ j _0 } - \beta_{ j_0 } = \lambda _ { j_0 } -
    \gamma _{ j _0 } =0$ and $j_1,j_2$ are in $ \supp ( \lambda -
    \alpha)$, we have that  $ \lambda - \beta$ and $ \lambda - \gamma$ 
    have strictly smaller supports than $ \lambda - \alpha$.
We observe that if  $
    \lambda _j < 0$ for some $j$, then this procedure will eventually produce
    vectors with $ \alpha _j ^ \ell = \lambda _j$ and thus satisfy \eqref{P}. 

We verify that the new vectors $ \beta$ and $ \gamma$ satisfy
\eqref{Index2}. From the definition of $ \beta $, it follows that
$$
\sum _ { j \neq j_0}( \lambda _j -\beta _j) = \sum _ { j =1} ^ N (
\lambda _j - \alpha _j)\geq 0
$$
since we have already observed that $ \sum _ { j =1 } ^N ( \lambda _j
-\alpha _j ) \geq 0$. Next we observe that the definition of $ \beta$
and \eqref{Index2} give
$$
\sum _ { j \neq j_1 } ( \lambda _j -\beta _j ) = ( \alpha_{ j_0 } -
\lambda _ { j _0 })+ \sum _ { j\neq j_1} (
\lambda _j - \alpha _j)  \geq 0.
$$
Finally, if $\ell$ is not $j_0$ or $j_1$, we have
$$
\sum _ { j \neq \ell } ( \lambda _j -\beta _j) = \sum _ { j \neq \ell } (
\lambda _j -\alpha _j ) \geq 0.
$$
The argument to show \eqref{Index2} for $\gamma $ is identical. 
\end{proof}

Thanks to Lemma \ref{Step}, we have reduced estimating the form $
\Lambda ( f_1, \dots , f_N)$ to estimating expressions of the form
$$
\Lambda ( |\cdot |^ { \alpha _1 } f_1, \dots, |\cdot |^ { \alpha _N }
f_N) 
$$
where the indices $( \alpha _ 1, \dots , \alpha _N) $ satisfy \eqref{P}, \eqref{P1} and
\eqref{Z}.

\begin{lemma}\label{Step1Lemma}
  Suppose that the indices $(1/p_1, \lambda _1, \dots, 1/p_N, \lambda _N) $
  satisfy the conditions (\ref{Scaling}),  (\ref{HyperStrict}), 
  (\ref{Index}), and the vector $( \alpha _1, \dots, \alpha _N) $ satisfies
  \eqref{P}, \eqref{P1} and \eqref{Z}. Then we have
  \begin{gather}\label{LAScale}
    \sum _ { j =1  }^N( \frac 1 { p_j } + 
 \frac { \lambda _j - \alpha _j } k )= 2,
\\\label{LAHyper}
  \sum _ { j \neq \ell } (\frac 1 { p_j } + \frac  { \lambda _ j - \alpha _j }{ k }) > 1, \qquad
  \ell = 1,\dots, N. 
  \end{gather}
\end{lemma}

\begin{proof}
  The identity \eqref{LAScale} follows from \eqref{Scaling} and
  \eqref{Z}.
  To  establish \eqref{LAHyper}, observe that using \eqref{Hyper} and  \eqref{Z}, we
  obtain
  $$
  \sum _ { j \neq \ell } (\frac 1 { p _j } +  \frac { \lambda _j - \alpha _j } k )
  = \frac{ \alpha_\ell } { k } +\sum _{ j \neq \ell } ( \frac 1 { p _j
  } + \frac{ \lambda _j} k). 
  $$
If we have $ \lambda _\ell \geq 0$, then \eqref{P1} implies $ \alpha _
\ell \geq 0 $ and \eqref{LAHyper} follows for such $ \ell$. In the case
that $ \lambda _ \ell <0$, we use \eqref{Scaling}  to obtain
$$
\frac{ \alpha_\ell } { k } +\sum _{ j \neq \ell } ( \frac 1 { p _j
  } + \frac{ \lambda _j} k) = 2 -( \frac 1 { p_\ell} +  \frac {
  \lambda _ \ell - \alpha _ \ell} k ) = 2 - \frac 1 { p_\ell} .
$$
The second equality follows from \eqref{P} and \eqref{P1} which gives $\lambda
_ \ell - \alpha _ \ell = 0$. 
Now \eqref{LAHyper}  follows since $ 1/p_ \ell
\in (0,1)$. 
\end{proof}

With these preliminaries out of the way, we are ready to
begin the proof of Theorem \ref{suff}. In fact, we  will prove a more general
theorem in Lorentz spaces,  Theorem
\ref{SuffLorentz} below. We note that 
Theorem  \ref{suff} then follows from Theorem \ref{SuffLorentz} by setting $p_j = r_j$.

\begin{theorem}\label{SuffLorentz} Suppose that the indices $(1/p_1, \lambda
  _1, \dots, 1/p_N, \lambda _N)  $ lie in $ ((0,1) \times \reals) ^N$ and
  satisfy \eqref{Scaling}, \eqref{HyperStrict}, and
  \eqref{Index}. Also, let the vector of indices $(1/r_1, \dots, 1/r_N)  \in [0,1]^N $ satisfy \eqref{Interpolation}. 
  Then we have
  $$
  \Lambda ( f_1, \dots, f_N) \leq C \prod _ { j =1 } ^N \| f _ j \| _
          { L^ { p_j, r _j } _ { \lambda _j }  }.
$$          
\end{theorem}

\begin{proof} According to Lemma \ref{Step},
  it suffices to consider expressions of the form
  $$
  I =\Lambda ( |\cdot |^ { \alpha _1 } f_1, \dots, |\cdot |^ { \alpha _N }
f_N) 
  $$
with $ ( \alpha _ 1, \dots, \alpha _N) $ satisfying \eqref{P}, \eqref{P1} and
(\ref{Z}).

We define exponents
\begin{equation}
  \frac 1 { q_ j } = \frac 1 { p _j } + \frac { \lambda _ j - \alpha _
    j } k , \qquad j =1, \dots,N. 
\end{equation}
We claim that this vector of exponents $(1/q_1, \dots, 1/q_N)$ 
lies in the set $ \intr({\cal P}_0)$
where $ {\cal P }_0$ is the polytope from the proof of Theorem
\ref{BCCTCor}. 

We first show that $ 1/q_j \in (0,1)$ for all $j$. 
Since we assume that $  1/ p _j \in (0,1)$ and $ 1/q_j = 1 /p_j$ if $
\lambda _j \leq 0$, we have $ 1/q_j \in  (0,1)$ in this case. 
Thus it remains to show 
 that $ 1/q_j  \in (0,1)$ if $\lambda _j >0$.
If $ \lambda _j \geq 0$, we claim that 
$$
0 < \frac 1 { p _ j } + \frac { \lambda _ j - \alpha _ j } k \leq
\frac 1 { p _ j } + \frac {\lambda _ j} k  < 1 .
$$
The first two inequalities follow from \eqref{P}, \eqref{P1} and our assumption
that $ 1/p_j >0$. The last follows
by subtracting \eqref{HyperStrict} from \eqref{Scaling}. 
Thus we may use Theorem \ref{BCCTCor}  to conclude that
$$
I \leq C \prod _ { \{ j: \lambda _j \leq 0 \} } \|f_ j \| _ { L^ { p _ j,r_j }_{ \lambda _ j
  } } \cdot \prod _ { \{j: \lambda _j > 0 \} } \| f_j \|_ { L^ { q_ j , r _ j } _{
    \alpha _ j } }.
$$
Finally, the observation \eqref{LHCons}
gives
$
\| f_ j \| _ { L^ { q_j, r_j } _ { \alpha _ j } } \leq C \| f _ j \| _
   { L^ { p_j, r_j } _ { \lambda _ j }},
$ for $j$ with $ \lambda _j >0$ which 
 completes the proof. 
\end{proof}

\section{Proof of Theorem \ref{nec}}\label{necproof} In this section, we give necessary conditions for the estimate \eqref{goal}.  These conditions follow by examining the behavior of the form under scaling in $ \reals ^ { 2k}$ and along subspaces. To carry out this argument, we need to find functions $( \phi_1, \dots, \phi_N)$ for which the form is not zero and the norms in $L^{p_j} _{ \lambda_j}$ are finite. Due to the possible singularity of the power weight at the origin, we use functions that are supported away from the origin. This makes the arguments slightly more involved than in earlier works. 

The necessity of \eqref{Scaling} follows by rescaling the form in $ \reals ^ { 2k}$. 
We begin our proof  by  choosing a unit vector
$w_1\in\mathbb{R}^2$ with $w_1\cdot v_j \neq 0$ for  
$j=1,\dots,N$, and let $w_2\in\mathbb{R}^2$ be orthogonal to $w_1$. We
let $ \varepsilon >0$ to be fixed and  define a set
$S\subset\mathbb{R}^{2k}$ by
\begin{align*}
S=\{yw_1+\varepsilon y'w_2: 1\leq |y|\leq 2, 1\leq |y'|\leq 2,
y,y'\in\mathbb{R}^k\}.
\end{align*}
In the definition of $S$,  we use the notation $yw=(yw^1,yw^2)\in\mathbb{R}^k\times\mathbb{R}^k$, where
$y\in\mathbb{R}^k$ and $w=(w^1,w^2)\in\mathbb{R}^2$.

Since $w_1\cdot v_j$ is non-zero for each $j=1,\dots,N$, we may choose
$\varepsilon$ small and $c_0$
such that
$
1/c_0\leq |v_j\cdot x|\leq c_0$ for $ x\in S, 1\leq j\leq N$.
Recall that $v_j\cdot x=v^1_jx^1+v^2_jx^2$ where $v_j=(v^1_j,v^2_j)\in\mathbb{R}^2$
and $x=(x^1,x^2)\in\mathbb{R}^k\times\mathbb{R}^k$.

Now let $\varphi_j(t)=\chi_{[R/c_0,Rc_0]}(|t|)$ and observe that if
$x\in 
R\,  S:=\{Rx: x\in S\}$ for $ R>0$
then we have that 
\begin{align*}
\prod_{j=1}^N \varphi_j (v_j\cdot x) = 1, \quad x\in R\,  S.
\end{align*}
We have $ \|\phi_j\|_ { L^{p_j} _ {\lambda_j} } = C
R^ { k /p_j + \lambda _j }$ for all $R>0$. Thus if we assume that 
estimate
(\ref{goal}) holds, we will have
\begin{align*}
cR^{2k} = |R\,  S| & \leq \Lambda(\varphi_1,\dots,\varphi_N)\\
&\leq C\prod_{j=1}^N \|\varphi_j\|_{L^{p_j}_{\lambda_j}}\\
&=CR^{\sum_{j=1}^N (\frac{k}{p_j}+\lambda_j)}.
\end{align*}
Since the inequality above is true for $0<R<\infty$, we obtain  \eqref{Scaling}.

To establish the necessity of estimate \eqref{Hyper}, we fix $\ell$
and let $w$ be a unit vector that is perpendicular to $v_{\ell}$. We define a set 
$S_R\subset\mathbb{R}^{2k}$ by
\begin{align}
S_R=\{yv_\ell+Ry' w: 1\leq|y|\leq 2, 1\leq|y'|\leq 2, y,y'\in\mathbb{R}^k \}.
\end{align}
We observe that $|S_R|\geq cR^k$ and for $R$ large there
exists a constant $c_0$ such that
\begin{align*}
R/c_0\leq |v_j\cdot x|\leq c_0R, \quad j\neq \ell,\, x\in S_R,
\end{align*}
and
\begin{align*}
|v_\ell|^2\leq |v_\ell\cdot x|\leq 2|v_\ell|^2, \quad x\in S_R.
\end{align*}
Thus, if we define $\varphi_j$ by
$
\varphi_j(t)=\chi_{[R/c_0,Rc_0]}(|t|)$, for $ j\neq\ell$ 
and 
$
\varphi_\ell(t)=\chi_{[|v_\ell|^2,2|v_\ell|^2]}(|t|)$, and assume that 
the estimate
\eqref{goal} holds, 
we will have
\begin{align}
cR^k\leq|S_R|\leq \Lambda(\varphi_1,\dots,\varphi_N)\leq cR^{\sum_{j\neq\ell}\frac{k}{p_j}+\lambda_j},
\quad R\,\,\mbox{ large.}
\end{align}
This implies \eqref{Hyper}.

Next, we turn to the necessity of condition \eqref{Index}. The proof is similar to the argument that we used for \eqref{Hyper}, 
except that we test on a set of unit size, rather than a set of
diameter comparable to $R$. As in
the proof of \eqref{Hyper}, we begin by fixing a vector $v_\ell$ and then let $w$ be a non-zero vector perpendicular
to $v_\ell$. We choose $y_0\in\mathbb{R}^k$ with $|y_0|=1$ and define
\begin{align*}
S_R=\{yv_\ell +(Ry_0+y')w: |y'|<1,\ 1<|y|<2,\  y,y'\in\mathbb{R}^k\},\quad
0<R<\infty.
\end{align*}
Note that the $2k$-dimensional measure of $S_R$ satisfies $|S_R|\geq c$.

From the definition of $S_R$, we have that there is a constant $c_0$
so that 
\begin{align*}
|v_j \cdot x-R(v_j \cdot w) y_0|\leq c_0, \quad x\in S_R\mbox{ and } j
\neq \ell. 
\end{align*}
Going further, if we set $\varphi_j(t) = \chi_{[0,c_0]}(|t-R(v_j\cdot w)y_0|)$ for $j \neq \ell$,  we have
$\varphi_j(v_j\cdot x)=1$ if $x\in S_R$. In \eqref{Assumption} we assume  that $v_j$ is
not parallel with $v_\ell$. Thus  we  have $v_j\cdot w \neq 0$ for $j \neq \ell$  and  it follows
that there exists  $R_0$ so that $
\|\varphi_j\|_{L^{p_j}_{\lambda_j}}\leq C R^{\lambda_j}$ for $ R>R_0$.

For $j=\ell$, we have $|v_\ell|^2\leq |v_\ell\cdot x|\leq 2|v_\ell|^2$, $x\in S_R$.
Thus if we put $\varphi_\ell(t)=\chi_{[|v_\ell|^2,2|v_\ell|^2]}(|t|)$,
then we have $\varphi_\ell(v_\ell\cdot x)=1$, $x\in S_R$ and 
$\|\varphi_\ell\|_{L^{p_\ell}_{\lambda_\ell}}\leq C$. Altogether, 
under the assumption the estimate \eqref{goal} holds, we obtain 
\begin{align}
c\leq |S_R| \leq \Lambda(\varphi_1,\dots,\varphi_N)\leq CR^{\sum_{j\neq \ell}\lambda_j}.
\end{align}
Letting $R\rightarrow \infty$  we obtain \eqref{Index}.

Finally, we consider the condition \eqref{Interpolation}. To begin, we
define a  function $\varphi _j$ by 
\begin{align*}
\varphi_j(t)=\sum_{m\geq 1} a_{j,m}2^{-m(\frac{k}{p_j}+\lambda_j)}\chi_{[2^m,2^{m+1}]}(|t|).
\end{align*}
A calculation shows that 
\begin{align*}
\int_{\reals^k} \varphi_j(t)^{ p_j} |t|^ { p_j\lambda _j }\,dt= c\sum_{m\geq 1}a_{j,m}^{ p_j}.
\end{align*}
Thus if we let $a_{j,m}=m^{-\frac{(1+\varepsilon)}{p_j}}$
for some $\varepsilon>0$, we have $\varphi_j\in L^{ p_j}$.

We choose $w_1\in\mathbb{R}^2$ a unit vector with $w_1\cdot v_j \neq 0$ for
$j=1,\dots,N$ and let $w_2$ be perpendicular to $w_1$. Fix $
\varepsilon >0$ and define
\begin{align*}
S_1 = \{w_1y + \varepsilon w_2 y':  
1\leq |y|\leq 2, 1\leq |y'|\leq 2,\ y,y'\in \mathbb{R}^2\}.
\end{align*}
Since $w_1\cdot v_j \neq 0$, $j=1,\dots, N$, we may choose
$\varepsilon$ small and find $c_0$ so that
\begin{align*}
1/c_0\leq |v_j\cdot x|\leq c_0.
\end{align*}

Finally, we set $S_m=2^m\cdot S$. If we have the estimate \eqref{goal},
then we will have 
\begin{align*}
\infty >  \Lambda(\varphi_1, \dots,\varphi_N) &\geq \sum_{m\geq 1}\int_{S_m} \prod_{j=1}^m
\varphi_j(v_j\cdot x)\, dx\\
&\geq c\sum_{m\geq 1} m^{-(1+\varepsilon)\sum_{j=1}^N \frac{1}{p_j}}.
\end{align*}
Since this estimate holds for all $\varepsilon>0$, it follows
that we must have $\sum_{j=1}^N \frac{1}{p_j}\geq 1$.\hfill $\Box$

\section{Application and Examples} 
\label{Apps}
We close by giving an application of our results to the study of weighted estimates for a multi-linear fractional integral. We fix $N$ distinct non-zero real numbers $\theta_1, \dots \theta _N$ and $ \lambda \in (0, k)$ and   define an $N$-linear operator by 
\begin{equation}\label{n2operator}
T_{N,\lambda} (f_1, \dots, f_N) (x) = \int _ { \reals^k} \frac 1 { |y|^ \lambda } \prod _ { j=1} ^N f_j (x- \theta _j y) \, dy .
\end{equation}
When $N=1$ (and $\theta _1=1$) this operator is the standard fractional integral or Riesz potential. The multi-linear version appears to have been introduced by Grafakos in \cite{MR1164632} who considers $L^p$-estimates.  
To study the operator \eqref{n2operator} we introduce the $N+2$-linear form
\begin{equation}\label{n2form}
 \Lambda _{N}( f_0, f_1,  \dots, f_N, f_{N+1})
 =\int _ { \reals^{ 2k}} f_{ N+1} (y) f_0(x)\prod _ { j =1} ^ N f_n ( x-\theta _1y) \, dy\, dx.
\end{equation}
This notation should not be confused with the notation $\Lambda _k$ that was only  used  in Proposition \ref{KEquals1}. 
Note that the collection of vectors $(e_1, e_1-\theta_1 e_2, \dots , e_1-\theta _N e _2, e_2)$ satisfies the assumption \eqref{Assumption} and thus the results of this paper apply to the form \eqref{n2form}. 
For $ 1 \leq p_j < \infty$, the estimate $\| T_{N, \lambda }  (f_1, \dots, f_N) \|_ { L^ { p_0'}_ { - \lambda _0}} \leq C \prod_ { j=0} ^ N \| f_j \| _ { L^ { p_j } _{ \lambda _j}} $ is equivalent to the following estimate for the form \eqref{n2form} with $ f_{N+1}= |\cdot |^ { - \lambda}$
\begin{equation} \label{n2estimate}
  | \Lambda _N ( f_0, \dots , f_N, |\cdot|^ { -\lambda} )|  \leq C \prod _ { j=0 } ^ N \|f_j \|_{ L^ { p_j } _{ \lambda _j }}. 
\end{equation}

Before giving our results for the form \eqref{n2form}, we note that Kenig and Stein \cite{MR1682725} and Moen \cite{MR3130311} have studied the operator $T_{N,\lambda}$ as a map into $L^p$-spaces with $p<1$.  Our use of duality to relate an estimate for an operator with an estimate for a form means that we are not able to study estimates with $p<1$. 

      \begin{theorem} \label{MLFIThm} Suppose that $\lambda$ and $(1/p_0, \lambda _0, \dots, 1/p_N, \lambda _N)$ satisfy the conditions: 
\begin{align}
&  \frac \lambda k + \sum _ { j=0} ^ N (\frac { 1 } {p_j} + \frac {\lambda _j} k )= 2,  \label{MLFIScale} \\
&  0<\lambda /k  <1, \\
 & \frac 1 { p_j }+ \frac  { \lambda _j }  k < 1, \  k=0, \dots, ,N,\label{gives1.9} \\
&  \sum _ { j =0}^N \frac 1 { p _j }  \geq 1, \\
 & \sum _ { j =0 } ^ N \lambda _j \geq 0 , \qquad \lambda + \sum _ { j \neq \ell } \lambda _j > 0,\   \ell = 0, \dots, N. \label{AnotherStrict}
\end{align}
Under these conditions, there exists a constant $C$ so that we have
the estimate \eqref{n2estimate}.
      \end{theorem}
It is tempting  to prove Theorem \ref{MLFIThm} by noting that $|\cdot |^ { -\lambda } $ lies in the Lorentz space $ L^\infty_ \lambda$. However, the case $p= \infty$ is not included in our Theorem \ref{SuffLorentz}. As a substitute, we observe that $ |\cdot |^ {- \lambda }$ belongs to the family of spaces $ L^ { p, \infty} _ \alpha$ with $ \alpha /k + 1/p = \lambda$ which allows us to  use the estimate  given in  Theorem \ref{SuffLorentz}. The assumption of a strict inequality in \eqref{AnotherStrict} is used when we  carry out this argument. Additionally,
note that condition \eqref{gives1.9} implies \eqref{HyperStrict}.
      \begin{proof}   
        We set $m = \min \{ \lambda + \sum _ { j \neq \ell} \lambda _j : \ell = 0, \dots N\}$. By our assumption \eqref{AnotherStrict} we have $m>0$ and thus we may choose $p_{ N+1}$ to satisfy $ 0 <  k /p_{ N+1} < \min(m,k, \lambda)$ and subsequently define $ \lambda _{ N+1}= \lambda - k / p _ { N+1}$. With these choices, Theorem \ref{SuffLorentz}  and the observation that
$ \| |\cdot |^ { -\lambda }\| _ {L^ { p_{ N+1}}_{\lambda _ {N+1}}}< \infty $ combine to give the Theorem. 
      \end{proof}

      We now show how Theorem \ref{MLFIThm} allows us to reprove several earlier results. Grafakos \cite[Theorem 1]{MR1164632} establishes the estimate \eqref{n2estimate} with $\lambda _j=0$, $j=0, \dots, N$ and 
      $$\frac 1 { p_j } \in [0,1), \qquad  \frac \lambda k + \sum _{ j=0}^N \frac 1 { p_j } =2. 
        $$
These cases are covered by Theorem \ref{MLFIThm}.     Note that Grafakos allows $p_j= \infty$, which can be handled by considering forms with a lower order of multi-linearity.

        The weighted, linear ($N=1$) case was treated by Stein and Weiss \cite[Theorem B*]{MR0098285} under the assumptions
        \begin{gather*}
          \frac 1 { p_0} + \frac 1 { p_1} + \frac { \lambda_0  + \lambda_1  +
    \lambda  } k = 2, \quad
  \frac 1 {p_j} \in (0,1),  \ j=0,1,\qquad 
  \frac \lambda k \in ( 0,1),     
\\   \frac 1 { p_0 } + \frac 1 { p_1} \geq 1, \quad
  \frac { \lambda _j }
  k + \frac 1 { p_j} < 1, \ j=0,1,  \qquad
    \lambda _0 + \lambda _ 1 \geq 0.
\end{gather*}
The scaling equality \eqref{MLFIScale} and the conditions $ \lambda _j/k + 1/p_j <1$ imply  $ (\lambda _j + \lambda ) /k >1/p_j' >0$ for $j=0,1$  so that Stein and Weiss's  conditions imply the strict inequality in condition \eqref{AnotherStrict}  in Theorem \ref{MLFIThm}. Thus, Theorem \ref{MLFIThm} gives a new proof of this classical  result.

As a final example, we consider several recent weighted results for the bilinear operator.   Hoang and Moen \cite[Theorem 10.1]{MR3776027} use general weighted estimates for bilinear fractional integrals to derive results with power weights. Their estimates do not appear to be optimal and later work by Komori-Furuya \cite[Theorem 2]{MR4047697} gives an improvement.  Komori-Furuya's conditions are 
\begin{gather*} 
  \frac  { \lambda} k + \sum _ { j=0}^2( \frac 1 {p_j} + \frac { \lambda _j } k ) =2, \quad \frac 1 { p _j } +\frac {\lambda _j } k < 1,\  j =0,1,2,\quad \frac 1 { p_0} + \frac 1 { p_1} + \frac 1 { p_2} \geq 1, 
  \\
  \lambda _0+ \lambda _1 +\lambda _2 \geq 0,\  \lambda  + \lambda _1 + \lambda _2 \geq 0,\  \lambda _j \leq \lambda , j = 1,2.
\end{gather*}
(Komori-Furuya neglects to list the condition $ 1/p_0 + 1/p_1 + 1/p_2 \geq 1$, but it is used in the argument given in section Appendix 6 of his paper, so we have included it here.) 
Note that the conditions $ \lambda _0+ \lambda _1 + \lambda _2 \geq 0$ and $ \lambda _j \leq \lambda $, $j=1,2$ imply $ \lambda + \lambda _0 + \lambda _j\geq 0$, $ j=1,2$. Thus, the result of Komori-Furuya will follow from Theorem \ref{MLFIThm}, except when we  have equality in   $ \lambda + \lambda _0 + \lambda _1 + \lambda _2 -\lambda _j \geq 0$ for some $j$ in $\{0,1,2\}$.  To give an explicit example, the vector of indices $(1/p_0, \lambda _0, 1/p_1, \lambda _1, 1/p_2 , \lambda _2, 0, \lambda )=(7/12, -1/8, 7/12, -1/8, 7/12, 1/4, 0, 1/4)$ will satisfy Komori-Furuya's conditions, but fails the strict inequality $ \lambda _0 + \lambda _1 + \lambda >0$ that we require. 
On the other hand  Theorem \ref{MLFIThm} includes sets of indices that are not covered by Komori-Furuya's results.   Again we give an explicit example: 
$ (1/p_0, \lambda _0, 1/p_1, \lambda _1, 1/p_2 , \lambda _2, 0, \lambda) =(1/2, 1/12, 1/2, 1/12, 1/2, 1/4,0,  1/12)$. Specifically, these indices fail the condition $\lambda _2 \leq \lambda$ of Komori-Furuya.
Our approach has the advantage of extending easily to the $N$-linear fractional integral. One might also argue that the set of conditions in Theorem \ref{MLFIThm} has the advantage of being more symmetric.

We close by giving a few examples that indicate that there is more work to do to
understand the behavior of the form on the boundary of the set
determined by the inequalities in Theorems \ref{nec} and
\ref{suff}. These examples are  for the form $\Lambda _1$ as defined in \eqref{n2form}, 
where it is easy to make calculations.

We consider the estimate
\begin{equation}\label{SW2} 
\int _ { \reals ^ { 2k}} f_1(x) f_2 (x-y) f_3 (y) \, dx\,dy
\leq C\prod _ { j=1} ^3 \| f_j \| _ { L^ { p_j} _ {\lambda _j} }.
\end{equation}
For our first example, fix
$(1/p_1, \lambda _1,1/p_2,\lambda_2,1/p_3,\lambda_3 ) = 
(1,0, 0, \lambda
_2, 1- \lambda _2 /k , 0)$. 
Noting that
\begin{equation} \label{SimplEst}
  |f( x) |\leq \| f \| _ { L^ \infty_ { \lambda
  }} |x|^ { -\lambda },
\end{equation}
 the failure of \eqref{SW2} for
this family of exponents  can be found, for example, in \cite[p.~119]{ES:1970}. However, if we consider the family of exponents,
$(1/p_1 , \lambda _1,1/p_2,\lambda_2,1/p_3,\lambda_3 ) = (1 , k-\lambda _2-\lambda _3,0, \lambda _2,0,\lambda
_3)$ with $ 0 < \lambda_2, \lambda _3 <k$ and $ \lambda _2 + \lambda
_3 >k$, the outcome is positive. Under this scenario,
using the estimate \eqref{SimplEst} for 
$ f_2$ and $f_3$ and elementary estimates (or 
the identity for the $k$-dimensional Beta function as in Stein
\cite[p.~118]{ES:1970}) gives 
$$
|\int_{ \reals ^k} f_2(x-y)  f_3(y) \, dy | \leq C\| f_2\| _{ L^ \infty _
  { \lambda _2}}
  \| f_3\| _{ L^ \infty _
    { \lambda _3}} |x|^ { k-\lambda _2-\lambda _3}.
    $$
Given this, the estimate \eqref{SW2} follows easily for this set of
exponents.

We close by listing a few avenues for further investigation.
\begin{enumerate}
\item Systematically study estimates for the forms treated in this paper on
    the boundary of ${\mathcal P}$ as defined in \eqref{setP}.
\item  Study estimates for more general Brascamp-Lieb forms in 
$L^p$-spaces with power weights.
  \item Consider estimates in weighted $L^p$-spaces with more general
    weights.
    
\end{enumerate}



\def\cprime{$'$} \def\cprime{$'$} \def\cprime{$'$} \def\cprime{$'$}
  \def\cprime{$'$} \def\cprime{$'$} \def\cprime{$'$} \def\cprime{$'$}
  \def\cprime{$'$} \def\cprime{$'$}

\end{document}